\documentclass[reqno]{amsart}
\usepackage{color}
\usepackage{graphicx}
\usepackage{tikz-cd}

\def\R{\mathbb{R}}
\def\N{\mathbb{N}}
\def\Z{\mathbb{Z}}
\def\d{{\mathrm d}}

\newcommand*{\eps}{\varepsilon}
\newcommand*{\re}{{\mathrm{e}}}
\newcommand*{\ri}{{\mathrm{i}}}

\newtheorem{theorem}{Theorem}
\newtheorem{corollary}[theorem]{Corollary}
\newtheorem{lemma}[theorem]{Lemma}
\newtheorem{proposition}[theorem]{Proposition}
\theoremstyle{remark}
\newtheorem{remark}[theorem]{Remark}

\title[Quasi-convergence of optimal balance by nudging]%
{Quasi-convergence of an implementation of optimal balance by
backward-forward nudging}

\author{G\"okce Tuba Masur}
\author{Haidar Mohamad}
\author{Marcel Oliver}

\address[G.T. Masur]{Institut für Atmosph\"are und Umwelt \\
  Goethe-Universit\"at Frankfurt \\
  60438 Frankfurt/Main \\
  Germany}

\address[H. Mohamad and M. Oliver]{Mathematical Institute for Machine
  Learning and Data Science \\
  KU Eichst\"att--Ingolstadt \\
  85049 Ingolstadt \\
  Germany}

\address[H. Mohamad and M. Oliver]{School of Engineering and Science \\
  Jacobs University \\
  28759 Bremen \\
  Germany}

\date{\today}
\subjclass[2020]{Primary 34E13; Secondary 34B15, 37M21}

\begin{document}

\begin{abstract}
Optimal balance is a non-asymptotic numerical method for computing a
point on an elliptic slow manifold for two-scale dynamical systems
with strong gyroscopic forces.  It works by solving a modified
differential equation as a boundary value problem in time, where the
nonlinear terms are adiabatically ramped up from zero to the fully
nonlinear dynamics.  A dedicated boundary value solver, however, is
often not directly available.  The most natural alternative is a
nudging solver, where the problem is repeatedly solved forward and
backward in time and the respective boundary conditions are restored
whenever one of the temporal end points is visited.  In this paper, we
show quasi-convergence of this scheme in the sense that the
termination residual of the nudging iteration is as small as the
asymptotic error of the method itself, i.e., under appropriate
assumptions exponentially small.  This confirms that optimal balance
in its nudging formulation is an effective algorithm.  Further, it
shows that the boundary value problem formulation of optimal balance
is well posed up at most a residual error as small as the asymptotic
error of the method itself.  The key step in our proof is a careful
two-component Gronwall inequality.
\end{abstract}

\maketitle

\section{Introduction}

Optimal balance is a non-asymptotic numerical method for computing a
point on an elliptic slow manifold for two-scale dynamical systems
with strong gyroscopic forces.  It is based on adiabatically deforming
the system from the full nonlinear vector field to a linear vector
field where the fast-slow splitting can be inferred from the spectrum
of the associated linear operator.  So long as the homotopy between
linear and nonlinear vector field approximately preserves the level of
fast energy, this computation yields a highly accurate approximation
to a point on a nearly invariant slow manifold of the nonlinear
system.  The method is applicable for slow manifolds in the general
sense of MacKay \cite{MacKay04} which may not be exactly invariant as
in classical singular geometric perturbation theory
\cite{DeMaesschalckK:2020:GevreyAP, Fenichel:1979:GeometricSP}.  In
particular, normal hyperbolicity is not required.

While the underlying ideas are much older, optimal balance was
introduced by Vi\'udez and Dritschel \cite{Viudez:04} in the context
of semi-Lagrangian schemes for geophysical fluid flow
\cite{dritschel1997contour,dritschel1999contour}.  The authors
observed excellent nonlinear separation of (slow) Rossby and (fast)
gravity waves.  Cotter \cite{cotter2013data} subsequently pointed out
the connection to the theory of adiabatic invariance, investigated
earlier for a finite dimensional toy model in
\cite{CotterR:2006:SemigeostrophicPM}.  Gottwald \emph{et al.}\
\cite{GottwaldMO:2017:OptimalBA} provide a detailed analysis of the
asymptotic error of optimal balance in the same finite dimensional
setting, but with a finite time horizon for the homotopy between the
nonlinear and linear system.  In this case, additional order
conditions on the ramp function appear and prevent the use of analytic
ramp functions.  Nonetheless, when the ramp function is of Gevrey
class $2$, the balance error can be exponentially small.

The practical numerical implementation of optimal balance requires
solving a boundary value problem in time.  The boundary conditions
are, respectively, absence of fast linear modes at the linear end of
the homotopy, and the requirement that a complementary set of
variables, a ``basepoint coordinate'' parametrizing the slow manifold,
takes a specified value at the nonlinear end.  When the dynamical
system is low-dimensional, a standard solver for boundary value
problems, sometimes even a simple shooting scheme, can be used.  For
large or complex models, this is not practical.  However, optimal
balance can be implemented via backward-forward nudging where the
homotopy is repeatedly integrated backward or forward in time between
the linear and full nonlinear state.  At each temporal boundary, the
respective boundary condition is imposed while the complementary
variables are left unchanged.  This is similar, even though the
precise details differ, to ``backward-forward nudging'' described by
\cite{AurouxN:2012:BackFN}, so that we borrow their terminology.  The
scheme as such was already used by \cite{Viudez:04} and found to work
well.  The key advantage of a backward-forward nudging implementation
of optimal balance is that any existing numerical code can be turned
into an optimal balance solver so long as the mode splitting of the
linearized system is understood, which opens possibilities for
accurate diagnostics even for complex operational atmosphere and ocean
models.

\begin{figure}[tb!] 
  \centering
  \includegraphics[width=0.8\textwidth]{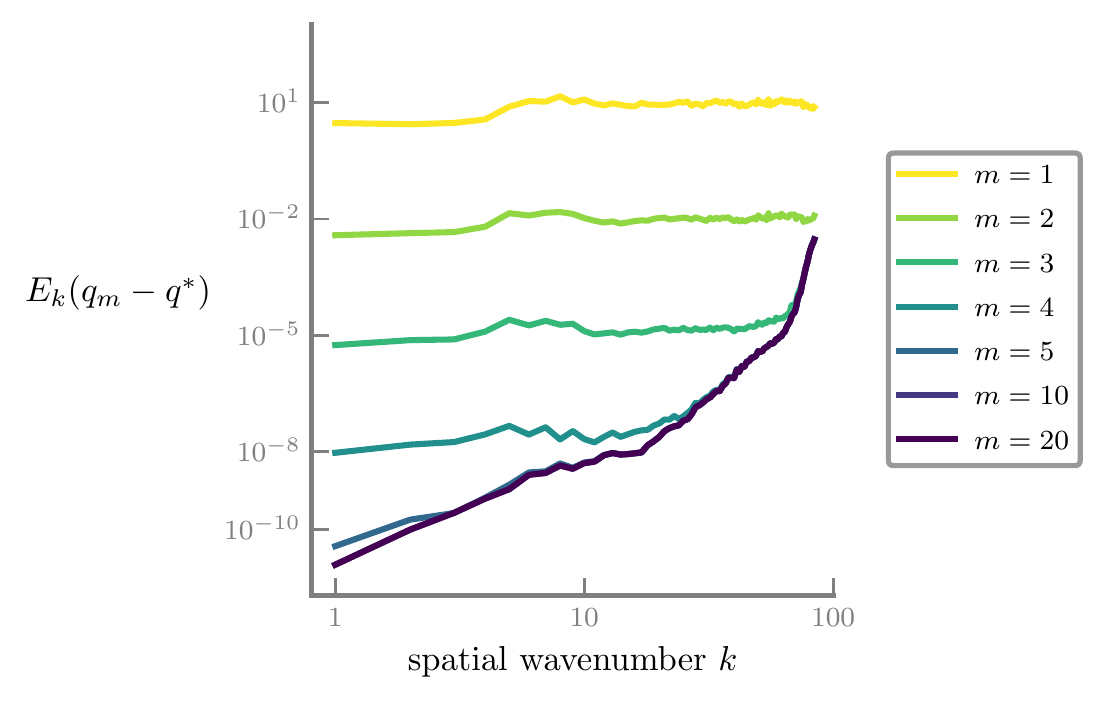}
  \caption{Decrease of the nudging error with the number of nudging
  iterations $m$ in an application of optimal balance to the rotating
  shallow water equations.  Shown is the spectral energy density $E_k$
  of the difference of successive nudging iterates $q_m$ from the
  given basepoint potential vorticity $q^*$ as a function of $k$, the
  modulus of the wavenumber vector.  The Rossby number of this
  simulation is $\eps=0.1$.  Adapted from \cite{Masur2022}.}
  \label{f.spectra}
\end{figure}

Two theoretical questions, however, were left open so far.  First, can
we prove that the optimal balance boundary value problem is well
posed?  Second, if so, can we prove that the implementation by nudging
actually converges to the solution of this boundary value problem? A
detailed numerical study of optimal balance for the rotating shallow
water equations in their primitive variable formulation
\cite{Masur2022,MasurOliver2020} found that, while the method does
return well-balanced states as expected, convergence of the sequence
of nudging iterates to a given basepoint takes place only up to a
small residual that does not go away as the number of iterations grows
large. Figure~\ref{f.spectra} shows typical diagnostic output of a
simulation of a two-dimensional shallow water flow, where the
basepoint variable $q^*$ is the potential vorticity field.  For
selected members of the sequence of nudging iterates, indexed by $m$,
we display the nudging residual as the difference between $q_m$ and
the prescribed basepoint $q^*$ in terms of its spectral energy density
$E_k(q_m-q^*)$, where the scalar total wavenumber $k$ is defined as
the modulus of the two-dimensional wavenumber vector.  The best
residual is reached within five iterations, subsequent iterations show
no further improvement.  The termination residual is very small for
low spatial wavenumbers and grows somewhat larger toward high spatial
wavenumbers.

In this paper, we revisit the issue mathematically in the simpler
context of the finite dimensional model problem that was already used
in related previous studies \cite{CotterR:2006:SemigeostrophicPM,
GottwaldOT:2007:LongTA, GottwaldO:2014:SlowDD}.  Without assuming
well-posedness of the boundary value problem, we split the error into
the termination residual of the iteration and the ``balance error'',
the residual fast energy of the optimal balance formulation as
analyzed in \cite{GottwaldMO:2017:OptimalBA}.  We prove that the
termination residual is small and of the same order of magnitude as
the balance error.  This result, first, provides a rigorous foundation
that optimal balance in a backward-forward nudging implementation is
indeed a very effective algorithm.  Second, it highlights that care
must be taken to find a good termination criterion for the nudging
iteration, cf.\ the discussion in \cite{Masur2022,MasurOliver2020}.
Third, it shows that we can side-step the question of well-posedness
of the optimal balance boundary value problem, assumed but not proved
in \cite{GottwaldMO:2017:OptimalBA}: our result implies that the
problem is well posed at worst up to a residual of the same small
order of magnitude than the overall error of the method.

For small problems, such as the numerical test case studied in
\cite{GottwaldMO:2017:OptimalBA}, it is possible to apply a proper
boundary value problem solver.  The results obtained there suggest
that at least in some cases, the boundary value problem is actually
well-posed and the termination residual can be brought to zero.  It is
likely, though, that the class of problems for which quasi-convergence
in the sense of this paper holds is strictly larger than the class of
problems for which the boundary value problem is well posed in a
strict sense. 

The remainder of the paper is structured as follows.  In
Section~\ref{s.model}, we introduce the model equations and sketch the
standard asymptotic construction of the slow manifold.
Section~\ref{s.ob} introduces optimal balance and extends the
asymptotic construction to the non-autonomous optimal balance
equations.  The key estimates from \cite{GottwaldMO:2017:OptimalBA}
are reviewed in Section~\ref{s.estimates} and adapted to the choice of
parameters required in the context of this paper.  Our new results are
contained in the main Section~\ref{s.nudging}: We introduce the
nudging scheme and estimate the termination residual with a careful
Gronwall estimate over the entire backward-forward integration cycle.
The main result, quasi-convergence of the scheme, is stated as
Theorem~\ref{e.maintheoremAleg} or~\ref{e.maintheoremExp} for the
respective case that algebraic or exponential order conditions are
satisfied.  The paper concludes with a discussion of scope and
possible improvements to our results.  A short appendix recalls a
Gronwall inequality for systems of differential inequalities, one of
the main ingredients in our argument.

\section{The model}
\label{s.model}

As in \cite{CotterR:2006:SemigeostrophicPM, GottwaldOT:2007:LongTA,
GottwaldMO:2017:OptimalBA,GottwaldO:2014:SlowDD}, we consider the
finite-dimensional toy model
\begin{subequations}
  \label{e.finite_Hamilt}
\begin{align}
  \dot{q} & = p \,, \\
  \eps \, \dot{p} & = Jp - \nabla V(q) \,,
\end{align}
\end{subequations}
where $q \colon [0,T] \rightarrow \mathbb{R}^{2d}$ is a vector of
positions $p \colon [0,T] \rightarrow \mathbb{R}^{2d}$ the vector of
corresponding momenta, $\eps$ is the time-scale separation parameter
considered to be small, $J$ is the standard symplectic matrix on
$\R^d$ with $d$ even, and $V$ is a sufficiently smooth potential.

For the purpose of defining and analyzing optimal balance, this system
is particularly easy because the slow subspace of the \emph{linear}
system, i.e.\ when $V=0$, is given by $p=0$.  This, however, is not a
true restriction but rather a convenient choice of variables.  A
hierarchy of slow manifold in the sense of \cite{MacKay04} is then
given by the finite power series expansion
\begin{equation}
  \label{e.G_series}
  p_{\text{slow}} \equiv G_n(q)
  = \sum_{i=0}^{n} g_i(q) \, \eps^i \,,
\end{equation}
where the coefficient vector fields $g_i$ are recursively defined by
\begin{equation}
  \label{e.recursive_g}
  \begin{aligned}
    g_0(q) &= - J \nabla V(q)\,, \\
    g_k(q) & = -J \sum_{i+j=k-1} Dg_i(q) \, g_j(q) \,.
  \end{aligned}
\end{equation}
The series generally does not converge, but it is very easy to prove
that solutions to $\dot q = G_n(q)$ are shadowed by a solution of
\eqref{e.finite_Hamilt} for finite $n$ and finite times; see, e.g.,
\cite{GottwaldOT:2007:LongTA}.  System \ref{e.finite_Hamilt} is
Hamiltonian so that, provided $V$ is analytic, Hamiltonian normal form
theory gives exponentially close shadowing over exponentially long
times \cite{CotterR:2006:SemigeostrophicPM}.

\section{Optimal balance}
\label{s.ob}

Optimal balance provides a numerical procedure to compute an
approximation to the map \eqref{e.G_series} for a given ``basepoint''
$q$ without the need to analytically compute any of the terms on the
right hand side. It works by selectively turning off the nonlinear
term in the equation via a ``ramp function''
$\rho \colon [0,1] \rightarrow [0,1]$ with $\rho(0)=0$ and
$\rho(1)=1$.  In addition, we assume that either
$\rho\in C^{n}([0,1])$ and satisfies the \emph{algebraic order
condition} of order $n \in \N^*$,
\begin{equation}
  \rho^{(i)}(0)=\rho^{(i)}(1)=0
  \quad \text{for } i=1,\dots,n \,,
  \label{e.alg-order}
\end{equation}
alternatively, that $\rho$ is of Gevrey class~2 on $[0,1]$ and
satisfies the \emph{exponential order condition}
\begin{equation}
  \rho^{(i)}(0)=\rho^{(i)}(1)=0
  \quad \text{for all } i \geq 1 \,.
  \label{e.exp-order}
\end{equation}
Recall that a function $f \in C^\infty(U)$ for $U \subset \R$ open is
of Gevrey class $s$, we write $f\in G^s(U)$, if there exist constants
$C$ and $\beta$ such that
\begin{equation}
  \sup_{x\in U} \big| f^{(n)}(x)\big| \leq C \, \frac{n !^s}{\beta^n}
  \label{Gevrey}
\end{equation}
for all $n\in \N$.  Here, we need this estimate to be uniform up to
the boundary of the interval $[0,1]$.  An example of a ramp function
satisfying this condition is
\begin{equation}
  \label{e.rhoexp}
  \rho(\theta) 
  = \frac{e^{-1/\theta}}{e^{-1/\theta} + e^{-1/(1-\theta)}} \,.
\end{equation}
The method of optimal balance requires solving the boundary value
problem in time,
\begin{subequations}
  \label{e.rHamilt}
\begin{align}
  \dot{q} & = p \,, \\
  \eps \, \dot{p} & = Jp -\rho(t/T) \, \nabla V(q) \,,
\end{align}
with boundary conditions
\begin{equation}
  \label{e.rHamilt_bcs}
  p(0)=0 \quad \text{and} \quad q(T) =q^*\,.
\end{equation}
\end{subequations}
The boundary condition at the linear end where $t=0$ expresses the
absence of fast motion, the boundary condition at the nonlinear end
where $t=T$ expresses the matching with a prescribed basepoint
$q^*$.  The approximation to the slow manifold $G$ is then given by
\begin{equation}
  G(q^*)=p(T) \,.
\end{equation}

Optimal balance is nicely illustrated by applying it to the following
classical example of an exact local slow manifold.  This example is
used, for instance, by MacKay \cite{MacKay04} as a starting point for
showing that generic Hamiltonian systems do not possess exactly
invariant slow manifolds.

Take a slow harmonic oscillator, written in action-angle variables
$(I, \theta)$ so that the action $I \equiv 1$ is constant and the
angle $\theta \in \R/2\pi$ has constant rate of change which drives,
via a nonlinear coupling function $f(\theta)$, a fast harmonic
oscillator written in complex representation, namely
\begin{subequations}
  \label{e.example}
\begin{gather}
  \dot \theta = 1 \,, \\
  \eps \, \dot p = \ri p + f(\theta) \,.
\end{gather}
\end{subequations}
The form of this equation is different from \eqref{e.finite_Hamilt},
but can be brought into a more similar form by a change of variables
the details of which do not matter for the point to be made.
Expanding $f$ as a Fourier series,
\begin{equation}
  f(\theta) = \sum_{k \in \Z} f_k \, \re^{\ri k \theta} \,,
\end{equation}
and inserting the ansatz $p=G(\theta)$ into \eqref{e.example}, we find
that the slow manifold is given by
\begin{equation}
  G(\theta)
  = \sum_{k \in \Z} f_k \, \frac{\eps}{\ri(k\varepsilon - 1)} \,
      \re^{\ri k \theta}
  \label{e.direct}
\end{equation}
away from the resonances $\eps = 1/k$.  The corresponding ramped
system for the optimal balance boundary value problem reads
\begin{subequations}
\begin{gather}
  \dot \theta = 1 \,, \\
  \eps \, \dot p = \ri p + \rho(t/T) \, f(\theta) 
  \label{e.example.b} \,.
\end{gather}
\end{subequations}
Using the optimal balance boundary conditions \eqref{e.rHamilt_bcs},
with $\theta$ in place of $q$, we can readily write out the solution
of \eqref{e.example.b} as
\begin{gather}
  p(T)
  = \int_0^T \re^{\ri(T-t)/\eps} \, \rho(t/T) \, f(\theta^*+(t-T))
    \, \d t \,. \label{e.pT} 
\end{gather}
To see that this expression is indeed a good approximation of
$G(\theta^*)$, we must insert the Fourier expansion for $f$ and
integrate by parts:
\begin{align}
  p(T)
  & = \sum_{k \in \Z} f_k
      \int_0^T \re^{\tfrac{\ri (T-t)}\eps + \ri k(\theta^*+(t-T))} \,
      \rho(t/T) \, \d t
      \notag \\
  & = \sum_{k \in \Z} f_k \, \re^{\tfrac{\ri T}\eps + \ri k(\theta^*-T)}
      \int_0^T \re^{(\tfrac{-\ri }\eps + \ri k)t} \,
      \rho(t/T) \, \d t
      \notag \\
  & = \sum_{k \in \Z} f_k \, \re^{\tfrac{\ri T}\eps + \ri k(\theta^*-T)} \,
      \frac{\eps}{\ri (k\eps - 1)}
      \biggl[ 
        \re^{(\tfrac{-\ri }\eps + \ri k)T}
        - \int_0^T  \re^{(\tfrac{-\ri }\eps + \ri k)t} \,
      \frac{\rho'(t/T)}T \, \d t
      \biggr] \,.
  \label{e.ibp}
\end{align}
In the last equality, we have used that $\rho(0)=0$ so that the
boundary term at $t=0$ vanishes.  The contribution from the first term
in brackets coincides with \eqref{e.direct}.  The contribution from
the remainder integral is, at this point, $O(\eps)$ provided that
$\eps$ is suitably bounded away from resonances.  However, its order
in $\eps$ can be improved by further integration by parts: So long as
$\rho^{(i)}(0)=\rho^{(i)}(T) = 0$ for $i= 1, \dots, n$, the boundary
terms of the $(n+1)$st integration by parts vanish and the
contribution from the remaining integral is $O(\eps^{n+1})$.  Thus,
\begin{equation}
  p(T) = \sum_{k \in \Z} f_k \, \frac{\eps}{\ri (k\eps - 1)} \,
         \re^{\ri k\theta^*} + O(\eps^{n+1}) \,.
\end{equation}
To get exponential estimates, some assumption on the growth of
derivatives of $\rho$ is necessary.  The corresponding analysis could
proceed along the lines of \cite{GottwaldMO:2017:OptimalBA}.

We conclude that optimal balance is consistent in the case where an
exact slow manifold (locally) exists, but introduces a small error
even there.  The interesting case, however, is when an exactly
invariant slow manifold does not exist, e.g., when the action of the
slow oscillator in \eqref{e.example} is not constant in time but
varies across one or more of the resonances.  Then, especially for
Hamiltonian systems, exponentially accurate normal forms are the best
we can expect.  Optimal balance works in both cases and can be
exponentially accurate, whether or not an exactly invariant slow
manifold exists.

\section{Remainder estimates}
\label{s.estimates}

The analysis of optimal balance in \cite{GottwaldMO:2017:OptimalBA}
works by constructing an asymptotic expansion analogous to
\eqref{e.G_series} for the ramped system.  This expansion defines a
time-dependent manifold on which the non-autonomous dynamics of
\eqref{e.rHamilt} remains predominantly slow.

The explicit form of the expansion is only used for the theoretical
analysis, and reads
\begin{equation}
  \label{e.F_series}
  p_{\text{slow}} \equiv F_n(q,t)
  = \sum_{i=0}^{n} f_i(q,t) \, \eps^i \,,
\end{equation}
where the coefficient vector fields $f_i$ are recursively defined by
\begin{equation}
  \label{e.recursiveq}
  \begin{aligned}
    f_0(q,t) &= - \rho(t/T) \, J \nabla V(q)\,, \\
    f_k(q,t) & = -J\partial_tf_{k-1}
    -J \sum_{i+j=k-1} Df_i(q,t) \, f_j(q,t) \,,
  \end{aligned}
\end{equation}
for $k=1,\dots,n$.  An $n$-term approximation to the fast component
of the motion is then given by
\begin{equation*}
  w(t) = p(t) - F_{n}(q,t)
\end{equation*}
(we suppress the dependence of $w$ on $n$ to keep the notation simple), and
the model \eqref{e.rHamilt} can be rewritten in $q$-$w$ variables as
\begin{subequations}
  \label{e.fs_ode}
  \begin{align}
    \dot{q} & = F_{n} + w \,, \label{eq:fs_ode_q} \\
    \dot{w} & = \Bigl( \frac{1}{\eps}J - DF_{n} \Bigr) \, w
      + R_{n}(q) \label{eq:fs_ode_w} \,,
  \end{align}
with remainder given by
\begin{equation}
  R_{n}(q)
  = - \eps^{n} \, \partial_t f_{n}(q)
    - \sum_{s = n}^{2n} \eps^s
      \sum_{\substack{i+j =s\\ i,j \leq n}} D f_i(q) f_j(q) \,
\end{equation}
(we suppress the explicit dependence of $R_n$, $F_n$, and $f_i$ on
time $t$ to keep the notation simple), and where the boundary
conditions \eqref{e.rHamilt_bcs} now read
\begin{equation}
  \label{e.fs_ode_bcs}
  w(0)=0 \quad \text{and} \quad q(T)=q^*.
\end{equation}
\end{subequations}

In the following, we review estimates on the slow vector field and on
the remainder of the asymptotic series.  When the series defined in
(\ref{e.F_series}--\ref{e.fs_ode}) is truncated at a fixed order $n$,
these estimates are straightforward and stated as
Lemma~\ref{LemAlgFDFR} below.  When we go for exponential estimates,
we need an optimal truncation of the asymptotic series.  Full details
can be found in \cite{GottwaldMO:2017:OptimalBA};
Lemma~\ref{LemExpFDFR} below states these results in the form required
here, together with only a sketch of the proof.

\begin{lemma} \label{LemAlgFDFR}
Let $B(0, r) \subset \R^{2d}$ for some $r>0$ and fix $T_1>0$.  Assume
that $\rho\in C^{n}([0,1])$ and $V\in C^{n+1}(B(0, r))$ for some
$n \in \N$.  Then there exist constants $C_1$, $C_2$, and $C_3$
such that for any $0<\eps <T \leq T_1$ and $t \in [0,T]$,
\begin{subequations}
  \label{AlgFDFR}
\begin{gather}
  \|F_{n}(q_1,t)- F_{n}(q_2,t)\| \leq C_1 \, \|q_1- q_2\| \,,
  \label{AlgF} \\
  \|DF_{n}(q_1,t) - DF_{n}(q_2,t)\| \leq C_2 \, \|q_1- q_2\| \,,
  \label{AlgDF}
\intertext{and}
  \|R_{n}(q,t)\| \leq C_3 \, \Bigl( \frac{\eps}T \Bigr)^{n}
  \label{AlgR}
\end{gather}
\end{subequations}
for any $q_1$,  $q_2$, and $q \in B(0, r)$.  The constants may depend
on $\rho$, $V$, $n$, and $T_1$, but are independent of $\eps$ and $T$.
\end{lemma}

\begin{proof}
Estimates \eqref{AlgF} and \eqref{AlgDF} follow directly from the
local Lipschitz property of $V$ and its higher order derivatives.
Note that the terms of $F_k$ contain powers of $T^{-1}$ at most up to
order $k$.  Thus, all constants in such terms can be bounded by
$C(T_1) \, (\eps/T)^k \leq C(T_1)$.  To prove \eqref{AlgR}, it
suffices to factor out $(\eps/T)^{n}$ from $R_{n}(q)$, then use the
available bounds on $\rho$ and $V$ and their derivatives together with
the assumption $\eps < T< T_1$.
\end{proof}

To obtain exponential estimates, we assume that $V$ is analytic and
$\rho$ is of Gevrey class~2. 

\begin{lemma}\label{LemExpFDFR}
Fix $r > r' >0$ and $T_1>0$. Assume that $\rho \in G^2(0,1)$ and $V$
is analytic on $B(0, r)$. Then there exist positive constants
$C_1$, $C_2$, $C_3$, $c$, and $\gamma$, each depending only on $\rho$, $V$,
and $T_1$, such that if
\begin{equation}
  n = \Big\lfloor \frac{\gamma \, \eps}T \Big\rfloor \,,
\end{equation}
then
\begin{subequations}
  \label{ExpFDFR}
\begin{gather}
  \|F_{n}(q_1)- F_{n}(q_2)\| \leq C_1 \, \|q_1- q_2\| \,,
  \label{ExpF}\\
  \|DF_{n}(q_1)- DF_{n}(q_2)\| \leq C_2 \, \|q_1- q_2\| \,,
  \label{ExpDF}
\intertext{and}
  \|R_{n}(q)\| \leq C_3 \, \re^{-c\sqrt[3]{\frac{T}{\eps}}}
  \label{ExpR}
\end{gather}
\end{subequations}
for $0 < \eps < T \leq T_1$ and all $q_1, q_2, q \in B(0, r')$.
\end{lemma}

\begin{proof}
For every $ n\in \N$,  \cite[Lemma~6]{GottwaldMO:2017:OptimalBA}
together with a Cauchy estimate asserts that 
\begin{subequations}\label{CauchyEstim}
\begin{gather}
  \|F_{n}(q_1)- F_{n}(q_2)\|
  \leq \frac{1}{r_1-r} \,
       \sup_{q \in B(0, r_1)}\|F_n(q)\| \, \|q_1- q_2\| \,, \\
  \|DF_{n}(q_1)- DF_{n}(q_2)\|
  \leq \frac{1}{(r_2- r)(r_3-r_2)} \,
       \sup_{q \in B(0, r_3)}\|F_n(q)\|\, \|q_1- q_2\| 
\end{gather}
\end{subequations}
for $r'< r_1, r_2 < r_3 < r$.  Fixing  $r_3$, we then need to prove
that there exists $n = n(\rho, V, T_1, \frac \eps T)$  such that
\eqref{ExpR} is satisfied and $\sup_{q \in B(0, r_3)}\|F_n(q)\|$ is bounded 
independent of $n$.    We briefly sketch the
main ideas:  The  key observation is that each of the terms appearing
in the expression for $R_n$ and $F_n$ is a product of functions which
depend only on $\rho$ and functions which depend only on $V$.  Hence,
each can be written as an inner product of a coefficient vector
encoding all $\rho$-dependence with a coefficient vectors encoding all
$V$-dependence as follows:
\begin{subequations}
\begin{gather}
  R_{n} = J \sum_{k = n}^{2n} \frac{\eps^k}{T^k} \,
  \langle \mathcal{R}_{k+1}(\rho),  \mathcal{F}_{k+1} (V) \rangle
\intertext{and}
  F_{n} = \sum_{k = 0}^{n} \frac{\eps^k}{T^k} \,
  \langle \mathcal{R}_k(\rho),  \mathcal{F}_k(V) \rangle \,.
\end{gather}
\end{subequations}
A H\"{o}lder-like inequality will bound each inner product, so that we
can estimate each class of coefficients separately in its respective
norm.  Indeed, for the $\rho$-dependent vectors, we use the Gevrey
estimate \eqref{Gevrey},
\begin{equation}\label{rhoVecEstim}
  |\mathcal{R}_k(\rho)| \leq  C \frac{(k+1)!^2}{\beta ^{k+1}},
\end{equation}
where the constants $C$ and $\beta$ depend only on $\rho$ and $T_1$.
On the other hand,  using a Cauchy estimate,   there exist constants
$C,  \gamma >0$ depending on $V$  (and implicitly on $B(0, r)$) such that 
\begin{equation}\label{VVecEstim}
  \|\mathcal{F}_k (V)\|_{B(0, r_3)}
  \leq C \, \Bigl( \frac{n}{\gamma} \Bigr)^k \,,
\end{equation}
where $\|\cdot\|_{B(0, r_3)} $ refers to the supremum norm on
$B(0, r_3)$.  Combining \eqref{rhoVecEstim} with \eqref{VVecEstim} and
using the Stirling inequality in the form
$m! \leq \re^{m-1} \, m^{m+ 1/2}$ for every $m \geq 2$, there exist
constants $C, \alpha >0$ depending only on $\rho$, $V$, and $T_1$ such
that
\begin{equation}
  \|R_{n}\|_{B(0, r_3)}
  \leq C \, \sum_{k = n}^{2n} \frac{\eps^k}{T^k} \, \frac{n^{3k}}{\alpha^k}
  \leq C \, \frac{\delta^{n}}{1-\delta} \,,
\end{equation}
where we have bounded the sum by the corresponding infinite geometric
series under the assumption that
$\delta \equiv \frac{\eps n^3}{\alpha T} <1$, and similarly
\begin{equation}
  \|F_{n}\|_{B(0, r_3)} \leq  C \, \frac{1-\delta^{n+1}}{1-\delta} \,.
\end{equation}
An optimization of the overall bound leads to the choice
\begin{equation}
  n = \Big\lfloor \Bigl(
        \frac{\alpha \delta T}{\eps}
      \Bigr)^{\frac 13} \Big\rfloor \,,
\end{equation}
so that
\begin{subequations}
\begin{gather}
  \|F_{n}\|_{B(0, r_3)} \leq \frac{C}{1-\delta}
\intertext{and}
  \|R_{n}\|_{B(0, r')}
  \leq \frac{C}{1-\delta} \, \delta^{(\alpha \delta T /\eps)^{\frac13}}
  \leq C_3 \, \re^{-c \sqrt[3]{\frac T \eps}}
\end{gather}
\end{subequations}
where, in the last inequality, we have fixed $\delta \in (0, 1)$ such
that the final constants are positive.
\end{proof}

\begin{remark}
The proof as written above differs from the proof given in \cite[Proof
of Theorem~4]{GottwaldMO:2017:OptimalBA} in two trivial respects: The
analysis in \cite{GottwaldMO:2017:OptimalBA} is done with respect to
the slow time $\tau \equiv \eps t$ and considers a sequence of ramp
times $T = 1/\eps$ for which $\rho(t/T) = \rho(\tau)$.  With this
choice, the ramp time $T$ does not appear in the formulas for
$\mathcal{R}_k(\rho)$.  Here, we use fast time $t$, so that
time-derivatives of $\rho(t/T)$ which appear in the expressions for
$\mathcal{R}_k(\rho)$ have coefficients proportional to $T^{-i}$,
$i = 0, \dots, k$.  To account for this, we changed the definition of
$\mathcal{R}_k(\rho)$, multiplying by $T^k$, so that the bound in
\eqref{rhoVecEstim} is uniform in $T \in (0, T_1]$.  The second
difference is that we estimate $F_n$ and $R_n$ on $B(0, r_3)$ and
$B(0, r')$, only requiring that $V$ is analytic on a larger ball of
radius $r > r_3>r'$ without an explicit requirement on their relative
sizes.
\end{remark}

\section{The nudging scheme}
\label{s.nudging}

To introduce the backward-forward nudging scheme, we write $q^*$ to
denote the prescribed basepoint coordinate, as before, and take an
initial guess $p_0$ for the value of the fiber coordinate
$p = G(q^*)$. E.g., $p_0=0$ or, if explicitly available,
$p_0=G_0(q^*)$, cf.\ \eqref{e.G_series}. We then construct a sequence
of approximates, $p_m$, in the following way.

Let $q^-_m$ and $p^-_m$ denote the solution to the ramped system
(\ref{e.rHamilt}a,b) \emph{backward} in time, endowed with the final
condition
\begin{equation}
  q^-_m(T) = q^* \quad \text{and} \quad p^-_m(T) = p_m \,.
\end{equation}
We stop the backward integration at $t=0$ and initialize a
\emph{forward} solution, $q^+_m$ and $p^+_m$, to (\ref{e.rHamilt}a,b)
with initial condition
\begin{equation}
  q^+_m(0) = q^-_m(0) \quad \text{and} \quad p^+_m(0) = 0 \,.
\end{equation}
This solution is stopped at $t=T$ and we set
\begin{equation}
  p_{m+1} = p^+_m(T) \,.
\end{equation}
In the following, we prove quasi-convergence of the sequence
$\{ p_m \}$ in the case of the algebraic order condition and the
exponential order condition.

To analyze the scheme, it is convenient to consider the ramped system
in terms of the slow-fast variables as in \eqref{e.fs_ode}, writing
$w^-_m(t)$ and $w^+_m(t)$ for the fast variable in the backward and
forward integration steps, respectively.  We emphasize that this
choice relates to the theoretical analysis of the method only.
Numerically, $w^-_m$ and $w^+_m$ are not available.  Similarly, we
write $w_m$ to denote the fast variable at the \emph{start} of the
backward-forward integration cycle.  I.e.,
\begin{equation}
  w_{m+1} = p^+_{m}(T) - F_n(q^*,T)
  \quad \text{for $m \geq 0$}
\end{equation}
and
\begin{equation}
  w_0 = p_0 - F_n(q^*,T) \,.
\end{equation}
(As before, we suppress the implicit dependence of $w_m$ on $n$.)

Figure~\ref{f.scheme} summarizes the nudging cycle which, in our
notation, starts at the top right and goes anticlockwise.  The
horizontal arrows indicate the backward/forward integration steps of
the ramped system along which the fast energy is adiabatically
preserved. The adjustment of the phase point at the linear end is of
size $\lVert w_m \rVert$ by construction.  At the nonlinear end, when
the basepoint is restored, the adjustment in the $q$ component, due to
the structure of equation \eqref{eq:fs_ode_q}, is very roughly of size
$\lVert w_m \rVert \, T$, so can be made small by choosing $T$
sufficiently small.  The precise statements and proofs are the
following.

\begin{figure}[tb!] 
  \centering
\begin{tikzcd}[column sep=6cm,
               row sep = huge]
  \text{\textbf{Linear end}} & \text{\textbf{Nonlinear end}} \\[-1cm]
  \begin{matrix}
    q_m^-(0) \\ p_m^-(0) \text{ discarded}
  \end{matrix}
  \arrow[d, "\text{Remove fast motion}" pos=0.33,
            "\text{here $p_m(0) = w_m(0)$ is fast}" pos=0.67] &
  \begin{matrix}
    q_m^-(T) := q^* \\
    p_m^-(T) :=  p_m
  \end{matrix}
  \arrow[l, "\text{Backward integration}"',
            "w_m^- \text{ almost constant}"] \\ 
  \begin{matrix}
    q_m^+(0) := q_m^-(0) \\
    p_m^+(0) := 0
  \end{matrix}
  \arrow[r, "\text{Forward integration}",
            "w_m^+ \text{ almost zero}"'] &
  \begin{matrix}
    q_m^+(T) \text{ discarded} \\  p_m^+(T) =: p_{m+1}
  \end{matrix}
  \arrow[u, "\text{Restore basepoint}" pos=0.67,
  "\text{Jump in $(q_m,w_m)$ small if $T$ small}" pos=0.33,
  "\text{$m := m+1$}"' pos=0.5]
\end{tikzcd}
\caption{Schematic representation of the nudging cycle. The arrow
annotations indicate the behavior of the implicit fast variable $w_m$
to explain the different steps of the proof of
Proposition~\ref{PropEstimSeqAlg}. }
\label{f.scheme}
\end{figure}
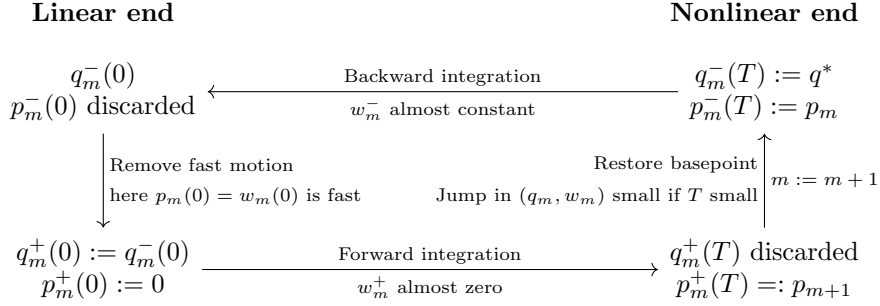

\begin{proposition} \label{PropEstimSeqAlg}
Assume that $\rho$ satisfies the algebraic order condition
\eqref{e.alg-order} of order $n \in \N$ at least at time $t=0$.  Fix
$R>0$, $T_1>0$, and $p_0, q^*\in \R^{2d}$ such that
\begin{equation}
  \label{e.ball}
  \max\big\{\|q^*\|, \|w_0\|\big\}
  \leq R \, \quad \text{for all } 0<\eps \leq T_1 \,.
\end{equation}
Then there exists $T_0 \in (0, T_1]$ and $\theta < 1$ such that for
any $0< \eps\leq T \leq T_0$, we have
\begin{equation}\label{wmMainEstim}
  \| w_{m+1} \| \leq \theta \, \| w_{m} \| + c \, \frac{\eps^n}{T^n} \,.
\end{equation}
The constants $\theta$ and $c$ may depend on $\rho$, $V$, $T_0$, and
$R$, but are independent of $T$ and $m$.
\end{proposition}

\begin{proof}
The proof consists of 4 steps.  We initially do our analysis on a
single backward-forward nudging cycle where we assume that the analog
of condition \eqref{e.ball} is satisfied at the start of the cycle,
i.e.,
\begin{equation}\label{wmR}
 \| w_m \| \leq R \,.
\end{equation}
In the final step, we prove that this conditions is maintained when
going from one nudging cycle to the next.

Step 1: Obtain a bound on the generation of fast motion during
backward integration.  We work under the assumption that there exists
$T_0 \leq T_1$ such that
\begin{equation}
  \label{e.bound}
  \sup_{t\in [0, T_0]} \max \{ \|q_m^\pm(t) \|, \|w_m^\pm(t) \| \}
  \leq 2R \,.
\end{equation}
For the backward integration, by continuity, there is always $T_0$
small enough such that this condition is satisfied.  Using the
equations for $w_m^\pm$ and $q_m^\pm$ and bounding the nonlinear terms
via \eqref{e.bound}, we see that there exist $C_1, C_2$, and $C_3$,
depending only on $\rho$, $V$, $n$, $T_1$, and $R$, such that
\begin{subequations}
  \label{bfrw}
\begin{align}
  \pm\frac{\d}{\d t}\|q_m^{\pm}\| & \leq C_1 + \|w_m^{\pm}\| \,,
  \label{bfq} \\
  \pm\frac{\d}{\d t}\|w_m^{\pm}\|
  & \leq C_2 \, \|w_m^{\pm}\| +C_3 \, \frac{\eps^{n}}{T^{n}} \,,
  \label{bfw}
\end{align}
\end{subequations}
Integrating \eqref{bfw} backward in time, we find
\begin{align}
  \label{bwm}
  \sup_{t\in [0, T]} \| w_m^-(t) \|
  & \leq \re^{C_2T} \, \|w_m \|
    + \frac{C_3}{C_2} \, (\re^{C_2 T} -1) \, \frac{\eps^n}{T^n}
    \notag \\
  & \leq \re^{C_2T} \, R + \frac{C_3}{C_2} \, (\re^{C_2 T} -1) \,
    \frac{\eps^n}{T^n} \,.
\end{align}
Inserting this bound into \eqref{bfq} and integrating backward, we
find
\begin{equation}\label{firstT0}
  \sup_{t\in [0, T]} \| q_m^-(t) \| \leq R + T \, \biggl( C_1 +
  \re^{C_2T} \, R
  + \frac{C_3}{C_2}\left( \re^{C_2 T} -1\right) \, \frac{\eps^n}{T^n}
  \biggr) \,. 
\end{equation}
Thus, choosing $T_0$ small enough, \eqref{e.bound} is maintained also
at the start of the forward integration step, with a right-hand bound
of $3R/2$, say.  By continuity, possibly lowering the value of $T_0$,
\eqref{e.bound} will be maintained over the forward evolution, too.
To get an explicit bound, we integrate \eqref{bfw} forward.  As $\rho$
satisfies the algebraic order condition \eqref{e.alg-order} of order
$n$ at $t=0$, we have $w_m^+(0)=0$ and therefore
 \begin{equation}\label{fwm}
  \sup_{t\in [0, T]} \| w_m^+(t) \|
  \leq \frac{C_3}{C_2} \left( \re^{C_2 T} -1\right) \,
  \frac{\eps^n}{T^n} \,. 
\end{equation}
Finally, inserting this bound into \eqref{bfq} and integrating
forward, we obtain
\begin{equation}
  \sup_{t\in [0, T]} \| q_m^+(t) \|
  \leq R + 2T \, \biggl( C_1 +
  \re^{C_2T} \, R
  + \frac{C_3}{C_2}\left( \re^{C_2 T} -1\right) \, \frac{\eps^n}{T^n}
  \biggr) \,. 
\end{equation}

Step 2: Obtain a bound on the deviation of $q$ from the basepoint
$q^*$ after completing one full cycle.  Taking the difference between
the backward and forward equation for $q$, we have
\begin{equation*}
  \frac{\d}{\d t}(q_m^+-q_m^-) =
  F_{n}(q_m^+,t)-F_{n}(q_m^-, t) + w_m^+-w_m^- \,.
\end{equation*}
It follows that
\begin{equation}
  \label{e.r1r2_ineq}
  \frac{\d}{\d t}\|q_m^+-q_m^-\| \leq
  \|F_{n}(q_m^+,t)-F_{n}(q_m^-, t) \| + \|w_m^+-w_m^-\| \,.
\end{equation}
Using estimate \eqref{AlgF},  we find that there exists $\delta_1 = \delta_1(\rho, V, n , T_0,  R)$ such that 
\begin{equation}
  \label{e.rr_estimate}
  \frac{\d}{\d t}\|q_m^+-q_m^-\|
  \leq \delta_1 \, \|(q_m^+-q^-) \| + \|w_m^+-w_m^-\| \,.
\end{equation}
In the same manner, taking the difference between the backward and
forward equation for $w$, we have
\begin{align}
  \frac{\d}{\d t}(w_m^+-w_m^-)
  & = \frac{1}{\eps} \, J(w_m^+-w_m^-)
      \notag \\
  & \quad - (DF_{n}(q_m^+,t)w_m^+ -DF_{n}(q_m^-,t)w_m^-)
      + R_{n}(q_m^+) - R_{n}(q_m^-) \,.
\end{align}
Adding and subtracting $DF_{n}(q_m^+,t)w_m^-$ from the right hand side
and taking the dot product with $w_m^+-w_m^-$, we obtain
\begin{align}
  \frac{\d}{\d t}\|w_m^+-w_m^-\| 
  & \leq \|DF_{n}(q_m^+,t)\| \, \|w_m^+-w_m^-\|
       + \|DF_{n}(q_m^+,t)-DF_{n}(q_m^-,t)\| \|w_m^-\|
    \notag \\
  & \quad + \|R_{n}(q_m^+)\|+ \|R_{n}(q_m^-)\| \,.
\end{align}
The Lipschitz property \eqref{AlgDF} and remainder bound \eqref{AlgR}
then imply that there exist constants $\delta_2$, $\delta_3$, and
$\alpha$, only depending on $\rho$, $V$, $n$, $R$, and $T_0$, such
that
\begin{equation}
  \label{e.ww_estimate}
  \frac{\d}{\d t} \|w_m^+-w_m^-\|
  \leq \delta_2 \, \|w_m^+-w_m^-\|
       + \delta_3 \, \|q_m^+-q_m^-\|
       + \alpha \, \frac{\eps^{n}}{T^{n}} \,.
\end{equation}
Thus, \eqref{e.rr_estimate} and \eqref{e.ww_estimate} form a system of
differential inequalities which can be written in matrix form as
\begin{equation}\label{e.zz_estimate_for1}
  \dot{z}(t) \leq \delta A \, z(t) + K \,,
\end{equation}
where $\delta = \max \{ \delta_1,  \delta_2,\delta_3 \}$,
\begin{gather}
  z(t) =
  \begin{pmatrix}
    \|q_m^+(t)-q_m^-(t)\| \\ \|w_m^+(t)-w_m^-(t) \|
  \end{pmatrix} \,,
  \quad 
  A = 
  \begin{pmatrix}
    1 & 1\\ 1 & 1
  \end{pmatrix} \,,
  \quad \text{and }
  K =
  \begin{pmatrix}
    0 \\  \alpha \, (\eps/T)^n
  \end{pmatrix} \,,
\end{gather}
and where we read the inequality sign component-wise.  A Gronwall
inequality for systems, recalled in the appendix below, then shows
that
\begin{align}
  z(t)
  & \leq z(0) + t K
      + \delta A \int_0^t\re^{\delta (t-s)A} \, (z(0) + sK) \, \d s
      \notag \\
  & = U_1(\delta, t) \, z(0) + U_2(\delta, t) \, K \,,
  \label{e.zz_estimate_for2}
\end{align}
where the second step follows from
$A \, \re^{\delta (t-s)A} = A \, \re^{2 \delta(t-s)}$ with
\begin{subequations}
\begin{gather}
   U_1(\delta, t) = I + \frac 12 \, (\re^{2 \delta t} - 1) \, A
\intertext{and}
   U_2(\delta, t) = t \, \left(I - \frac 12 A\right)
   + \frac{\re^{2\delta t} - 1}{4\delta} \, A \,.
\end{gather}
\end{subequations}
As $q^+(0) = q^-(0)$, we then obtain
\begin{align}
  \|q_m^+(T) - q^*\|
  & = \|q_m^+(T) - q_m^-(T)\|
      \notag \\
  & \leq \frac 12 \, (\re^{2 \delta T} - 1) \, \|w_m^-(0)\|
      + \alpha \, \frac{\re^{2\delta T} - 2\delta T -1}{4\delta} \,
        \frac{\eps^{n}}{T^{n}} \,.
  \label{qEstim1}
\end{align}

Step 3: Translate the mismatch of the basepoint $q$ into a bound on
the fast variable $w$ at the start of the next nudging cycle.  By the
triangle inequality,
\begin{align}\label{Iter_wm}
 \| w_{m+1} \|
  & = \|p_m^+(T) - F_n(q^*,T)\|
      \notag \\
  & \leq\|w_m^+(T)\| + \|F_n(q_m^+(T),T) - F_n(q^*,T)\|
      \notag\\
  & \leq C \, \|q_m^+(T) - q^*\| + c \, \frac{\eps^{n}}{T^{n}}
      \notag\\
  & \leq  C \, \re^{C_2T} \, (\re^{2 \delta T} - 1) \, \|w_m\|
      + c(T) \, \frac{\eps^{n}}{T^{n}} \,.
\end{align}
Here, we used Lemma~\ref{LemAlgFDFR} together with \eqref{fwm} in the
second inequality and estimate \eqref{qEstim1} together with estimate
\eqref{bwm} in the third inequality.
A possible explicit formula for $c(T)$ is
\begin{equation}
  c(T) = \alpha_1 \, (\re^{2 \delta T} - 2\delta T -1)
  + \alpha_2 \, (\re^{2 C_2 T} - 1)
\end{equation}
where $\alpha_i = \alpha_i(\rho, V, n , T_1, R)$.  Since $C$, $C_2$,
$\delta$, and the $\alpha_i$ are uniform in $m \in \N$ and
$T\in (0, T_0]$, estimate \eqref{wmMainEstim} is satisfied by a
further lowering of $T_0$.

Step 4: We finally prove that $T_0$ can be chosen such that
\eqref{wmR} remains satisfied in all nudging cycles provided it is
satisfied initially.  Estimate \eqref{Iter_wm} implies, in particular,
that
\begin{equation}
  \| w_{m+1} \|
  \leq C \, \re^{C_2T} \, (\re^{2 \delta T} - 1) \, R + c(T) 
\end{equation}
with $\lim_{T \to 0} c(T) = 0$.  Thus, for $T_0$ again small enough,
we ensure $\| w_{m+1} \| \leq R$, which completes the proof.
\end{proof}

We now turn to the analog of Proposition~\ref{PropEstimSeqAlg} when
$V$ is analytic and $\rho \in G^2(0, 1)$.  Here, the truncation order
$n$ must be chosen optimally.  The scheme is initialized with
\begin{equation}
  w_0 = p_0 - F_n(q^*, T) \,,
\end{equation}
which apparently depends on $n$. However, choosing $n$ as in
Lemma~\ref{LemExpFDFR}, $F_n(q^*, T)$ can be bounded uniformly in
$\eps < T < T_1$ for some $T_1>0$.  This bound depends on $T_1$,
$\rho$, and $V$, in particular on the domain on which $V$ is analytic.
Thus, for a given $R>0$, fix $r_i$, $i = 1, 2, 3$, such that
$ 2R \equiv r' < r_1, r_2 < r_3 < r \equiv 3R$ and estimates
\eqref{ExpFDFR} are satisfied.  Then, by the same steps as in the
proof of Proposition~\ref{PropEstimSeqAlg}, the following is true.

\begin{proposition} \label{PropEstimSeqExp}
Assume that $\rho$ satisfies the exponential order condition
\eqref{e.exp-order} at least at time $t=0$.  Fix $R>0$, $S>0$,
$T_1>0$, and $p_0, q^*\in \R^{2d}$ such that $V$ is
analytic on $B(0, 3R)$,
\begin{equation}
  \|q^*\| \leq R \text{ and } \|w_0\|
  \leq S \quad \text{for all } 0<\eps < T \leq T_1 \,,
\end{equation}
where $n$, implicit in the definition of $w_0$, is chosen as a
function of $\eps$ and $T$ as in Lemma~\ref{LemExpFDFR}.  Then there
exists $T_0 \in (0, T_1]$ and $\theta < 1$ such that for any
$0< \eps\leq T \leq T_0$, we have
\begin{equation}\label{wmMainEstimExp}
  \| w_{m+1} \| \leq \theta \, \| w_{m} \|
  + d \, \re^{-c \sqrt[3]{\frac T \eps}} \,.
\end{equation}
The constants $d$ and $c$ may depend on $\rho$, $V$, $T_0$,  $S$ and
$R$, but are independent of $T$ and $m$.
\end{proposition}


Propositions~\ref{PropEstimSeqAlg} and~\ref{PropEstimSeqExp} already
suffice to prove quasi-convergence of the nudging sequence.  However,
we would like to assert that the sequence of nudging iterates not only
quasi-converges to some limit, but that this limit is representative
of the slow manifold of the \emph{original} problem
\eqref{e.finite_Hamilt}, not that of the ramped problem
\eqref{e.rHamilt}.  This requires imposing an order condition not only
at $t=0$ but also at $t=T$.  Our final results can be stated in the
following form.

\begin{theorem}[Algebraic quasi-convergence]
\label{e.maintheoremAleg}
Assume that $\rho$ satisfies the algebraic order condition
\eqref{e.alg-order} of order $n \in \N$.  Suppose
$V \in C^{n+1}(B(0, 2R))$ for some $R>0$ and pick
$q^*, p_0 \in \R^{2d}$ and $T_1 >0$ such that
\begin{equation}
  \max \{ \|q^*\|, \|p_0 - G_n(q^*) \| \} \leq R
  \quad \text{for all } 0< \eps \leq T_1 \,.
\end{equation}
Choose $T_0 \in (0,T_1]$ such that the main estimate
\eqref{wmMainEstim} from Proposition~\ref{PropEstimSeqAlg} is
satisfied for all $\eps <T \leq T_0$.  Then there is a constant
$C = C(\rho, V, n, T, R)$ such that the sequence of nudging iterates
quasi-converges to a point $G_{n}(q^*)$ on the approximate slow
manifold of order $n$ in the sense that
\begin{equation}
  \label{e.pm_estimateAlg}
  \limsup_{m \to \infty} \|p_m-G_{n}(q^*)\| \leq C \, \eps^{n} \,.
\end{equation}  
\end{theorem}

\begin{proof}
The order condition \eqref{e.alg-order} implies that
$G_{n}(q^*) = F_{n}(q^*, T)$.  Thus, by
Proposition~\ref{PropEstimSeqAlg},
\begin{equation}
  \|p_{m+1}-G_{n}(q^*)\|
  = \|w_{m+1}\|
  \leq \theta \, \|w_m\| + c \, \frac{\eps^n}{T^n}
  = \theta \, \|p_m-G_{n}(q^*, T)\| + c \, \frac{\eps^n}{T^n} \,.
\end{equation}
Taking the limsup of this estimate, we find that
\begin{equation}
  \limsup_{m \to \infty} \|p_m-G_n(q^*) \| \leq
  \frac{c}{1- \theta} \, \frac{\eps^{n}}{T^n} \,,
\end{equation}
which concludes the proof.
\end{proof}

The corresponding statement for exponential quasi-convergence is the
following, its proof completely analogous to the previous.

\begin{theorem}[Exponential quasi-convergence]
\label{e.maintheoremExp}
Assume that $\rho$ satisfies the exponential order condition
\eqref{e.exp-order}.   Fix $R>0$, $S>0$,
$T_1>0$, and $p_0, q^*\in \R^{2d}$ such that $V$ is
analytic on $B(0, 3R)$,
\begin{equation}
  \|q^*\| \leq R \text{ and } \|p_0 - G_n(q^*) \| \leq S 
  \quad \text{for all } 0<\eps < T \leq T_1 \,,
\end{equation}
where $n$ is chosen as a function of $\eps$ and $T$ as in
Lemma~\ref{LemExpFDFR}.  Choose $T_0 \in (0,T_1]$ such that the main
estimate \eqref{wmMainEstimExp} from Proposition~\ref{PropEstimSeqExp}
is satisfied.  Then there are positive constants $C = C(\rho, V, T)$
and $ D = D(\rho, V, T)$ such that
\begin{equation}
  \label{e.pm_estimateExp}
  \limsup_{m\to \infty} \|p_m(T)-G_{n}(q^*)\|
  \leq D \, \re^{-C \eps^{-\frac 13}} 
\end{equation}
for all $\eps < T \leq T_0$.
\end{theorem}

\begin{remark}
  $G_n$ in \eqref{e.pm_estimateExp} is a truncated asymptotic series
  where the order of truncation grows like $\eps^{-1/3}$.  The
  exponent $\tfrac13$ may not be sharp, but its growth must be
  sublinear in $\eps^{-1}$, in contrast with the scaling of the
  optimal truncation for the original fast-slow system, which goes
  like $\eps^{-1}$ and yields an approximation to the slow manifold to
  $O(\exp(-c/\eps))$, see, e.g.,
  \cite{CotterR:2006:SemigeostrophicPM}.  $G_n$ is difficult to
  compute and need not be known when applying the optimal balance
  method in practice.  Still, \eqref{e.pm_estimateExp} expresses that
  the quasi-limit of the nudging sequence, i.e., the state returned by
  the optimal balance method, is exponentially well-balanced.
\end{remark}

\section{Discussion}

Our results show that, in order to guarantee quasi-convergence, one
first has to choose the ramp time $T$ sufficiently small.
Quasi-convergence then holds for all $\eps \leq T$, but to ensure
small remainders, we actually need $\eps \ll T$.  This may seem
restrictive, but in practice, reasonably large values of $T$,
comparable to the natural time scales of the slow motion, work just
fine \cite{Masur2022,MasurOliver2020,Viudez:04}.

One place where our estimates are over-conservative is the bound
\eqref{e.bound} which we apply equally to the $q$ and $w$ components
of the transformed system.  As the iterations progress, $w$ will get
successively smaller, so tracking separate bounds for $q$ and $w$ would
improve the constants.  Thus, it may be possible to start with $T$
small and increase it as the iterations progress.  This might improve
the basin of quasi-convergence as well as computational efficiency.
Alternatively, the basin of convergence may be extended by using
damped nudging updates of the form
$p_{m+1} = p_m + \alpha \, (p_m^+(T) - p_m)$ with $\alpha \in (0,1)$.
Here, once again, our practical experience is that the algorithm is
rather robust and does not usually require such measures, but there
may certainly be situations where they could help.

Our method of proof only gives an upper bound on the termination
residual, but does not prove that true convergence is impossible.  It
suggests, however, that quasi-convergence is the best we should hope
for, for the following reason: Both the termination residual and the
balance error arise from the same mechanism, the spurious excitation
of fast degrees of freedom via the late terms of a diverging
asymptotic series over an $O(1)$-interval of slow time.  As such, we
can control their amplitude, but their phases will typically depend
sensitively on the initial state of a nudging cycle, and have no
reason to converge.  Thus, the nudging iterates land in a ball whose
radius is controlled by the amplitude of the spurious fast motion
while the uncertainty of the landing point within the ball is due to
the unstable fast phases.  As a rule of thumb, we should expect
quasi-convergence for systems that do not possess truly invariant slow
manifolds.

The test problem used in this paper is a Hamiltonian system with a
noncanonical symplectic matrix which structurally resembles the
equations for rapidly rotating fluid flow, cf.\ the discussion in
\cite{CotterR:2006:SemigeostrophicPM}, albeit bypassing the
difficulties related to the functional setting and time horizon of
existence of solutions for the full fluid equations.  As such, it is a
paradigm for a class of systems of interest in geophysical fluid
dynamics.  Existence of approximately invariant slow manifolds for
certain Hamiltonian systems with finitely many slow degrees of
freedom, without restrictions on the number of fast degrees of
freedom, has been proved in \cite{KristiansenW:2016:ExponentialES}.
The class of systems considered there differs from our model problem
as the authors assume a symplectic matrix that is block-diagonal with
respect to the splitting into slow and fast variables.  However, the
only structural feature we really use is that the fast subsystem is
skew to leading order in the scale separation parameter.  We therefore
conjecture that optimal balance, as well as the analysis presented
here, extends not only to symplectic slow manifolds as in
\cite{KristiansenW:2016:ExponentialES}, but to a much larger class of
systems where, in particular, the slow subsystem is not structurally
constrained.  Including infinite dimensional dynamical systems will
bring its own set of challenges, but progress should be possible along
the lines of \cite{KristiansenW:2016:ExponentialES,
MohamadO:2018:HsCC, MohamadO:2019:DirectCS}.

Another point of interest is the choice of time integration scheme.
Solving the optimal balance boundary value problem \emph{exactly}
requires accuracy on the fast time scale, which can be very costly
when the scale separation is strong.  With the nudging solver, our
experience indicates that it suffices to choose a time integration
scheme that is stable with a time step selected to ensure accuracy on
the slow time scale, making the method more computationally feasible.
A precise understanding of this phenomenon may be possible along the
lines of this paper, but remains outside of the present scope.

\section*{Acknowledgments}

We thank the referees for careful reading and many insightful comments
that helped improve the presentation.  In particular, the suggestion
to include the example in Section~\ref{s.ob} and the conjecture that a
damped nudging update can improve the basin of convergence are due to
the referees.

This paper is a contribution to the Collaborative Research Center TRR
181 ``Energy Transfers in Atmosphere and Ocean'' funded by the
Deutsche Forschungsgemeinschaft (DFG, German Research Foundation)
under project number 274762653. GTM received further support through
the Collaborative Research Center TRR 301 ``The Tropopause Region in a
Changing Atmosphere'' funded by the Deutsche Forschungsgemeinschaft
(DFG, German Research Foundation) under project number 428312742. HM
was further supported by German Research Foundation grant OL-155/6-2.

\section*{Appendix}

In the following, we recall a Gronwall inequality for systems which
was proved in \cite{ChandraD:1976:LinearGG} and put this theorem into
a form which can be used to directly obtain estimate
\eqref{e.zz_estimate_for2}.

\begin{theorem} \label{t.chandra}
Fix $n \in N^*$ and let $G(t)$ and $H(t)$ be real-valued, continuous,
nonnegative $n\times n$ matrices.  Let further $z(t), a(t)\in \R^n$ be
continuous such that
\begin{equation}
  z(t) \leq a(t) + G(t) \int_0^t H(s) \, z(s) \, \d s \,.
\end{equation}
Then,
\begin{equation}
  z(t) \leq a(t) + G(t) \int_0^t V(t, s) \, H(s) \, a(s) \, \d s
\end{equation}
where 
\begin{equation}
  \label{Vttau}
  V(t, \tau) = I + \int_\tau^t H(s) \, G(s) \, V(s,  \tau) \, \d s \,.
\end{equation}
The inequalities above are satisfied component-wise.
\end{theorem}

Let $A$ denote the $n\times n$ matrix whose components are all ones
and let $\delta >0$.  It is easy to check that $\delta A$ is
nonnegative.  Indeed, given $X = (x_1, \dots, x_n)^T \in \R^n$, we
have
\begin{equation}
  X^T AX
  = \sum_{k=1}^n x_k \, (AX)_k
  = \sum_{k=1}^n x_k \sum_{k=1}^n x_j
  = \biggl( \sum_{k=1}^n x_k \biggr)^2
  \geq 0 \,.
\end{equation}
Then, writing
\begin{gather}
  G(t) = \delta A \,, \quad
  H(t) = I \,, \quad
  \text{and } V(t, \tau) = \re^{\delta(t-\tau) A} \,,
\end{gather}
we easily check that $V(t, \tau)$ satisfies \eqref{Vttau}.
Theorem~\ref{t.chandra} then implies the following.

\begin{corollary}
Let $\delta >0$, $K\in \R^n$, and $z(t) \in \R^n$ be differentiable
such that
\begin{equation}
  z'(t) \leq \delta A \, z(t) + K \,.
\end{equation}
Then
\begin{gather}
  z(t) \leq z(0) + t K
  + \delta A \int_0^t \re^{\delta (t-s)A} \, (z(0) + sK) \, \d s \,.
\end{gather}
\end{corollary}

\bibliographystyle{siam}
\bibliography{references}

\end{document}